\documentclass[12pt]{article}
\usepackage{amssymb,amsmath,amsthm}
\usepackage{graphicx}
\usepackage{subcaption}
\usepackage{fullpage}
\usepackage{scrextend}
\usepackage{siunitx}
\usepackage{comment}
\usepackage{tikz,calc}
\usepackage{tikz-cd}
\usetikzlibrary{patterns,shapes}
\usepackage{subcaption}
\usetikzlibrary{calc}
\usepackage{epsfig,graphicx}
\usepackage{listings}
\usepackage{enumerate}
\usepackage{enumitem}

\newtheorem{thm}{Theorem}[section]
\newtheorem*{theorem*}{Theorem}
\newtheorem{lem}[thm]{Lemma}

\newtheorem{ex}[thm]{Example}
\newtheorem{cor}[thm]{Corollary}

\newtheorem{prob}[thm]{Problem}

\title{On the number of distinct roots of a lacunary polynomial over finite fields}
\author{J\'ozsef Solymosi\thanks{solymosi@math.ubc.ca} \qquad Ethan P. White\thanks{epwhite@math.ubc.ca} \qquad Chi Hoi Yip \thanks{kyleyip@math.ubc.ca} \\
Department of Mathematics\\
The University of British Columbia \\
Vancouver, BC\\
 Canada V6T 1Z2}
\begin{document}

\maketitle

\abstract{We obtain new upper bounds on the number of distinct roots of lacunary polynomials over finite fields. Our focus will be on polynomials for which there is a large gap between consecutive exponents in the monomial expansion. 
}

\section{Introduction}\label{secintro}

A polynomial is \emph{lacunary} if there is a substantial gap between the degree of two consecutive terms. Most often, the gap between the highest and second highest terms is considered. What entails a substantial and useful gap depends on the context. The theory of lacunary polynomials has been critical in applications to computing theory, character sums, and discrete geometry.

Throughout this work, $q$ is a power of a prime $p$ and $\mathbb{F}_q$ is the field with $q$ elements. All polynomials considered will be in the ring $\mathbb{F}_q[x]$. For a polynomial $f(x) \in \mathbb{F}_q[x]$, we will denote by $f^\circ$ and $f^{\circ \circ}$ the degree of $f(x)$, and the degree of the second highest term of $f(x)$, respectively. For a polynomial $f(x) \in \mathbb{F}_q[x]$ we will denote by $Z(f)$ the set of roots of $f(x)$ in $\mathbb{F}_q^*$. A fundamental question is to improve the trivial degree bound on the number of distinct zeros of a polynomial. L\'aszl\'o R\'edei's monograph~\cite{r} is one of the most significant and important works on lacunary polynomials. One of the main theorems proved by R\'edei showed that polynomials cannot be too lacunary while also being fully reducible. For other similar results and applications of them, we refer to~\cite{BB}.

A class of polynomials related to lacunary polynomials are sparse polynomials. A $t$-sparse polynomial is a polynomial with $t$ terms in its monomial expansion. Sparse polynomials are referred to as lacunary by some authors. Analogous to the lacunary results mentioned above, Bibak and Shparlinski showed that few polynomials are simultaneously sparse and fully reducible~\cite{BS}. An upper bound on the number of roots of sparse polynomials has been investigated by several authors. For a $t$-sparse polynomial $f$, Karpinski and Shparlinski showed that $ |Z(f)| \leq \frac{t-1}{t}(q-1)$, and gave an efficient approximation algorithm for $|Z(f)|$, see~\cite{KS}. In Section~\ref{secimpdeg} we give a generalization of this bound. Canetti et al. proved a finite field analogue of Descartes' rule of signs using pigeonholing based on geometry of numbers and affine transformations on the exponents in the monomial expansion of a $t$-sparse polynomial in~\cite{CF}. Kelley~\cite{K} refined their method and showed that 
\begin{equation}\label{kelley} |Z(f)| \leq 2 (q-1)^{1-1/(t-1)} C(f)^{1/(t-1)},\end{equation}
where $f$ is a $t$-sparse polynomial and $C(f)$ denotes the size of the largest coset in $\mathbb{F}_q^*$ on which $f$ vanishes completely. In the case $t=3$, Kelley and Owen \cite{KO} improved the above bound on $|Z(f)|$ for trinomials $f$ to  
\begin{equation}\label{kelleyowen} D(f) \left\lfloor 
\frac{1}{2}+\sqrt{\frac{q-1}{d}}\right\rfloor, \quad \text{for } f(x) = x^n + ax^s + b, \quad \text{where } D(f)=\gcd(n,s,q-1).\end{equation}
Moreover, when $q$ is a square, they gave explicit examples for which the above bound is tight. The aforementioned bounds for sparse polynomials do not always improve the trivial degree bound for lacunary polynomials, since lacunary polynomials are not necessarily sparse, and vice versa. In Section~\ref{secimpdeg} our results improve the bounds in~\eqref{kelley} and~\eqref{kelleyowen} in the case that the polynomial is simultaneously sparse and lacunary. One of R\'edei's seminal results on lacunary polynomials is the following. 

\medskip
\begin{thm} [Theorem 5 in \cite{r}] \label{r}
Let $q$ be a prime power and $d>1$ be a divisor of $q-1$. Let $f(x) \in \mathbb{F}_q[x]$ be a monic polynomial such that 
\[ f(x) | (x^{q-1}-1), \quad f^\circ = \frac{q-1}{d}, \quad \text{and} \quad f^{\circ \circ} \leq \frac{q-1}{d^2}.\]
Then $f(x)$ is an Euler binomial
\[ x^{\frac{q-1}{d}} - \alpha, \quad \text{for some} \quad \alpha
\in (\mathbb{F}_q^\ast)^{\frac{q-1}{d}}
,\]
or if $p\neq 2$, $4|(q-1)$, and $d = 2$, then possibly takes the form
\[ \left( x^{\frac{q-1}{4}} - \beta \right) \left( x^{\frac{q-1}{4}} - \gamma \right) , \text{ where} \quad \beta^2 = 1, \gamma^2 = -1 .\]

\end{thm}

The above theorem shows that within a certain class of polynomials, a polynomial $f(x)$ cannot be simultaneously lacunary and possess $f^\circ$ distinct roots. Our main focus will be to extend this property to a larger class of polynomials and show that the number of distinct roots of many lacunary polynomials is often less than its degree. Hereafter $d \geq 1$ will always be a positive divisor of $q-1$. Our main focus will be on polynomials $f(x) \in \mathbb{F}_q[x]$ of the form 
\begin{equation}\label{fform} f(x) = x^{\frac{q-1}{d} - \ell} + g(x),\end{equation}
where $\ell \geq 0$ and $g^\circ < \frac{q-1}{d}-\ell$. 

\section{Main results}\label{secmain}

In this section we present our main results. The following theorem serves as the foundation for many of our more involved results.

\begin{thm}\label{t1} Let $\ell \geq 0$ be a nonnegative integer. Suppose $f(x) \in \mathbb{F}_q[x]$ has the form $x^{\frac{q-1}{d} - \ell} + g(x)$, for some $g(x) \in \mathbb{F}_q[x]$ such that $1 \leq g^\circ < \frac{q-1}{d} - \ell$. Let $\delta = \frac{q-1}{d} - \ell - g^\circ$, be the gap between the exponents of the two highest terms. Then  $|Z(f)| \leq d(\ell+g^{\circ}) = q-1 - d\delta$.
\end{thm}

The following is the main theorem of Section 3 and describes when we are able to guarantee that $|Z(f)|<f^\circ$ for $f(x)$ of the form~\eqref{fform}.

\begin{thm}\label{t3} Let $\ell \geq 0$ be a nonnegative integer. Suppose $f(x) \in \mathbb{F}_q[x]$ has the form $x^{\frac{q-1}{d} - \ell} + g(x)$, for some $g(x) \in \mathbb{F}_q[x]$ such that $1 \leq g^\circ < \frac{q-1}{d} - \ell$. If one of the following holds, then $|Z(f)| < f^{\circ}$.

\begin{enumerate}[label=\textbf{(\arabic*)}]

\item $d(d+1)\ell + d^2g^\circ < q-1 $;
\item $d^2(\ell + g^\circ) \leq q-1$ and $d(d+1)\ell>q-1$;
\item $d^2(\ell + g^\circ) > q-1$, $d\ell + d^3 g^\circ < q-1$, and $d(d^2+1) \ell + d^3 g^\circ < (q-1)(d+1)$. 
\end{enumerate}
\end{thm}

The inequalities in the cases within Theorem~\ref{t3} create regions on an $g^\circ,\ell$-axis system. See Figure~\ref{f1} for visualization of when $|Z(f)|<f^\circ$.

At the end of Section~\ref{secimpdeg} we present some results on $t$-sparse polynomials. Theorem~\ref{intthm} is a generalization of Theorem~\ref{t1}, giving a bound on $|Z(h)|$ for a $t$-sparse polynomial $h$. In Section 4 we refine the arguments used in Theorem~\ref{t3} to obtain lower upper bounds on $|Z(f)|$ for some $f(x)$ as in~\eqref{fform}. The following theorem is the main result of Section 4.

\begin{thm}\label{bestboundthm} Let $f(x) \in \mathbb{F}_q[x]$ be as in~\eqref{fform}. 

\begin{enumerate}[label=\textbf{(\arabic*)}]
\item If $\ell > \frac{q-1}{d(d+1)}$ and $i \geq -1$ is the largest integer such that 
\begin{equation*}\label{1eq} \ell + g^\circ < (q-1) \left( \frac{1+ d^{-2i-1}}{d(d+1)}\right),\end{equation*}
then
\[ |Z(f)| \leq \frac{q-1}{d+1} - d^{2i+2}\left(\ell - \frac{q-1}{d(d+1)}\right).\]

\item If $\ell + g^\circ  < \frac{q-1}{d(d+1)}$ and $i \geq -1$ is the largest integer such that 
\begin{equation*}\label{2eq} \ell > (q-1) \left( \frac{1- d^{-2i-2}}{d(d+1)}\right),\end{equation*}
then 
\[|Z(f)| \leq \frac{q-1}{d+1} - d^{2i+3} \left( \frac{q-1}{d(d+1)} - (\ell + g^\circ) \right).\]

\item  If $\ell \leq \frac{q-1}{d(d+1)}$, $ \ell + g^\circ \geq \frac{q-1}{d(d+1)}$, and $d(d+1)\ell + d^2 g^\circ < q-1$, then 
\[ |Z(f)| \leq d(\ell+g^\circ) .\]

\item If $\ell \leq \frac{q-1}{d(d+1)}$, $ \ell + g^\circ \geq \frac{q-1}{d(d+1)}$, and $d(d+1)\ell + d^2 g^\circ \geq q-1$, then 
\[ |Z(f)| \leq f^\circ = \frac{q-1}{d} - \ell .\]
\end{enumerate}

\end{thm}

%%%%%%%%%%%%%%%%%%%%%%%%%%%%%%%%%%%%%%%%%%%%%%%%
%%%%%%%%%%%%%%%%%%%%%%%%%%%%%%%%%%%%%%%%%%%%%%%%
%%%%%%%%%%%%%%%%%%%%%%%%%%%%%%%%%%%%%%%%%%%%%%%%
%%%%%%%%%%%%%%%%%%%%%%%%%%%%%%%%%%%%%%%%%%%%%%%%
%%%%%%%%%%%%%%%%%%%%%%%%%%%%%%%%%%%%%%%%%%%%%%%%

\section{Improving the degree bound}\label{secimpdeg}

Our main focus in this section is to identify lacunary polynomials $f(x) \in \mathbb{F}_q[x]$ that must have strictly less than $f^\circ$ distinct nonzero roots. As remarked upon in the introduction, we will consider polynomials $f(x) \in \mathbb{F}_q[x]$ of the form 
\begin{equation*}\label{fform2} f(x) = x^{\frac{q-1}{d} - \ell} + g(x),\end{equation*}
where $\ell \geq 0$ and $g^\circ < \frac{q-1}{d}-\ell$. Since we are interested in nonzero roots, we will always assume that the constant term of $f(x)$ is nonzero. Our first two theorems are motivated by the following well-known result in the case $d = 1$, we include a proof for completeness.

\begin{lem} Suppose $f(x) \in \mathbb{F}_q[x]$ has the form $x^m + g(x)$, for some $g(x) \in \mathbb{F}_q[x]$ such that $1 \leq g^\circ < m$. Let $\delta =m-g^\circ$ be  the gap between the exponents of the two highest terms. Then  $|Z(f)| \leq  q-1 -\delta$.
\end{lem}
\begin{proof} Note that for any $a \in \mathbb{F}_q^*$, $a^{q-1}=1$. Consequently, for any $a \in \mathbb{F}_q^*$ we have $a^{q-1-m}f(a) = a^{q-1-m}g(a)+1$. This gives
\[
|Z(f)|=|Z(x^{q-1-m}g(x)+1)|\leq q-1-m+g^\circ=q-1-\delta.
\]
\end{proof}

If the degree of the polynomial $f$ is bounded by $\frac{q-1}{d}$, then we have the following improved upper bound of $|Z(f)|$.

\medskip
\noindent\textbf{Theorem~\ref{t1}.}\begin{em} Let $\ell \geq 0$ be a nonnegative integer. Suppose $f(x) \in \mathbb{F}_q[x]$ has the form $x^{\frac{q-1}{d} - \ell} + g(x)$, for some $g(x) \in \mathbb{F}_q[x]$ such that $1 \leq g^\circ < \frac{q-1}{d} - \ell$. Let $\delta = \frac{q-1}{d} - \ell - g^\circ$, be the gap between the exponents of the two highest terms. Then  $|Z(f)| \leq d(\ell+g^{\circ}) = q-1 - d\delta$.
\end{em}

\begin{proof} Since $Z(f) = Z(x^\ell f)$, we consider the roots of $x^{\frac{q-1}{d}} + x^\ell g(x)$. A root of $x^\ell f(x)$ is a root of 
\begin{equation}\label{t1e} \xi + x^\ell g(x),\end{equation}
for some $\xi \in (\mathbb{F}_q^*)^{\frac{q-1}{d}}$. Since $g^\circ \geq 1$, (\ref{t1e}) has at most $g^\circ$ roots. Combining this with $|(\mathbb{F}_q^*)^{\frac{q-1}{d}}| = d$ gives the required bound. 

\end{proof}

For a polynomial $f(x)$ satisfying the hypotheses of Theorem~\ref{r} that is not a binomial, note that Theorem~\ref{t1} implies that $f^{\circ \circ}=\frac{q-1}{d^2}$. Below is an example when Theorem~\ref{t1} is tight. 

\begin{ex}\label{ex}\rm Let $p$ be a prime $p \equiv 7 \pmod {20}, p>7, d=2, \ell=1$ and $g^\circ=2$. By the law of quadratic reciprocity, we have 
\[
\bigg(\frac{5}{p}\bigg) \bigg(\frac{p}{5}\bigg)=(-1)^{2 \cdot \frac{p-1}{2}}=1. 
\]
Since $(\frac{p}{5})=(\frac{2}{5})=-1$, we have $(\frac{5}{p})=-1$, and so $5$ is a quadratic non-residue in $\mathbb{F}_p$. Then $-5$ is a quadratic residue in $\mathbb{F}_p$ since $p \equiv 3 \pmod 4$. Let $a \in \mathbb{F}_p$ be such that $-5a=1$, thus $a$ is a quadratic residue in $\mathbb{F}_p$. Let $S = \{1,4,4a\}$. Then $S$ is a subset of quadratic residues and $-S$ is a subset of quadratic non-residues. Define the polynomial 
\[f(x)=x^{\frac{p-1}{2}-1}-16ax^2-(4a+5) \in \mathbb{F}_p[x].\]
It is easy to check that $S^{-1} \cup (-S^{-1}) \subseteq Z(f)$. By Theorem \ref{t1}, $|Z(f)| \leq 6$, and so $Z(f) = S^{-1} \cup (-S^{-1})$ and Theorem~\ref{t1} is tight in this case. In particular, when $p=47$, we can take $S=\{1,4,18\}$ and $f(x)=x^{22}+22x^2+24$.
\end{ex}

We can combine Theorem \ref{t1} and the trivial degree bound on $|Z(f)|$ to obtain the following bound, which is independent of the divisor $d$.
\begin{cor}\label{sqrtcor}
Let $\ell \geq 0$ be a nonnegative integer. Suppose $f(x) \in \mathbb{F}_q[x]$ has the form $x^{\frac{q-1}{d} - \ell} + g(x)$, for some $g(x) \in \mathbb{F}_q[x]$ such that $1 \leq g^\circ < \frac{q-1}{d} - \ell$. Then $|Z(f)| \leq \sqrt{(q-1)(\ell+g^\circ)}$.
\end{cor}
\begin{proof}
Note that we have the trivial degree bound $|Z(f)|\leq f^\circ\leq \frac{q-1}{d}$, and by Theorem \ref{t1}, $|Z(f)|\leq d(\ell+g^\circ)$. Therefore, $|Z(f)| \leq \sqrt{\frac{q-1}{d} d(\ell+g^\circ)}=\sqrt{(q-1)(\ell+g^\circ)}$.
\end{proof}

In the above discussion, we restricted $\ell$ to be a nonnegative integer. When $\ell<0$, a similar trick leads to the following theorem.

\begin{thm}\label{t2} Let $m \geq 0$ be a nonnegative integer. Suppose $f(x) \in \mathbb{F}_q[x]$ has the form $x^{\frac{q-1}{d} + m} + g(x)$, for some $g(x) \in \mathbb{F}_q[x]$ such that $1 \leq g^\circ < \frac{q-1}{d} +m$. Then $|Z(f)| \leq d \max \{m,g^{ \circ}\}$. 
\end{thm} 

\begin{proof} If $x \neq 0$, then 
\begin{equation}\label{t2e}  f(x) = \xi x^m + g(x),\end{equation}
for some $\xi \in (\mathbb{F}_q^*)^{\frac{q-1}{d}}$. Since $g^\circ \geq 1$, and the constant term of $g$ is nonzero, then $\xi x^m + g(x)$ is a nonzero polynomial with degree at most $\max\{m,g^\circ\}$, so (\ref{t2e}) has at most $\max\{m,g^\circ\}$ roots. There are $d$ such $\xi \in (\mathbb{F}_q^*)^{\frac{q-1}{d}}$ and so the required bound follows.

\end{proof}

Below is an example when Theorem~\ref{t2} is tight. 

\begin{ex}\rm Let $p$ be a prime $p \equiv 3 \pmod {4}, p>3, d=2, m=1$ and $g^\circ=2$. Let $r_1,r_2 \in \mathbb{F}_p^*$ be any two quadratic residues. Note that $r_1 + r_2 \neq 0$ since $p \equiv 3 \pmod {4}$. Let $a \in \mathbb{F}_p$ be such that $a(r_1+r_2) = -1$. Define the polynomial 
\[ f(x) = x^{\frac{p-1}{2} + 1} + ax^2+ar_1r_2 \in \mathbb{F}_p[x]. \]
It is easy to check that $Z(f) = \{r_1,r_2,-r_1,-r_2\}$.

\end{ex}

We remark that in general the best known bounds on the number of zeros of a trinomial are due to Kelley and Owen~\cite{KO} and recorded in Equation~\eqref{kelleyowen}. Their bounds are on the order of $\sqrt{q}$ in magnitude. For trinomials satisfying the hypotheses of Theorem~\ref{t1} or Theorem~\ref{t2}, the respective theorems offer a significantly better bound on $|Z(f)|$. Examples for which Theorem~\ref{t1} or Theorem~\ref{t2} is tight seem harder to construct for larger $g^\circ$. The example below gives a class of examples for primes $p$ for which $-29$ is a square in $\mathbb{F}_p$.

\begin{ex}\rm Let $p$ be a prime $p \equiv 31 \pmod {116}, d=2, m=1$ and $g^\circ=4$.  Let $S=\{4,9,16,-29\}$, using a similar argument as in Example \ref{ex}, we can show $S$ is a subset of quadratic residues and $-S$ is a subset of quadratic non-residues. Define the polynomial 
\[ f(x) = 6500x^{\frac{p-1}{2} + 1} +(x-4)(x-9)(x-16)(x+29)-6500x \in \mathbb{F}_p[x]. \]
It is easy to check that $Z(f) = S \cup (-S)$.
\end{ex}

In the above example, $f(x)$ is a $4$-sparse polynomial with 8 distinct roots. Kelley's bound in Equation~\eqref{kelley} gives a bound on $|Z(f)|$ on the order $p^{2/3}$. Once again this demonstrates that a sparsity-only bound on $|Z(f)|$ such as~\eqref{kelley} can be significantly improved if $f$ is also lacunary.

We will see that the above two theorems can be combined and iterated to yield a stronger statement on the size of $|Z(f)|$. The coloured regions in Figure~\ref{f1} indicate when the degree bound on $|Z(f)|$ can be improved in terms of $\ell$ and $g^\circ$. We will prove the content of Figure~\ref{f1} in Theorem~\ref{t3}. The numbers on the coloured regions of Figure~\ref{f1} correspond to the cases described in Theorem~\ref{t3}. Next we present a proof of Theorem~\ref{t3}.

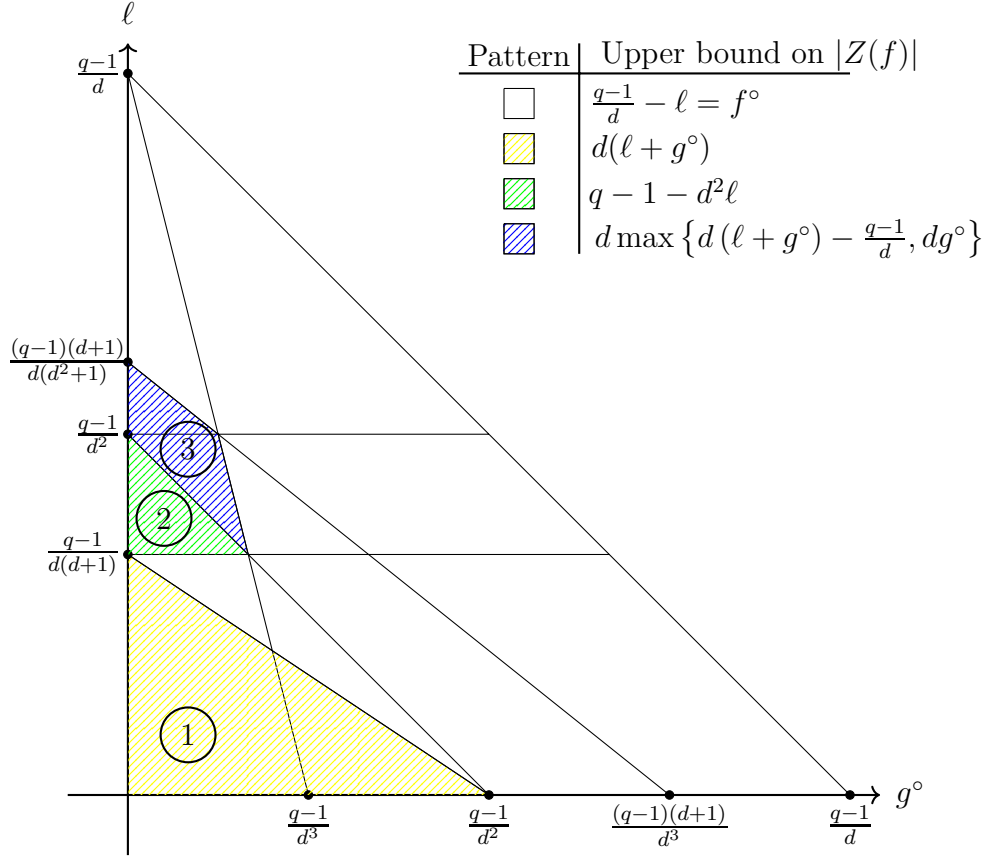
\begin{figure}[h!]
\begin{center}
\begin{tikzpicture}[scale=0.8]

\def\q{120}
\def\d{2}
\def\s{0.2}

\draw[->,thick]

(-1,0)--(12.5,0);

\draw[->,thick]

(0,-1)--(0,12.5);

\draw
(13,0) node{$g^\circ$}
(0,13) node{$\ell$}

(\s*\q/\d,-0.5) node{$\frac{q-1}{d}$}
(-0.5,\s*\q/\d) node{$\frac{q-1}{d}$}

(-1,{\q*\s*(\d+1)/(\d*(\d*\d+1))}) node{$\frac{(q-1)(d+1)}{d(d^2+1)}$}

({\s*\q*(\d+1)/(\d*\d*\d)},-0.5) node{$\frac{(q-1)(d+1)}{d^3}$}

(\s*\q/\d^2,-0.5) node{$\frac{q-1}{d^2}$}
(-0.5,\s*\q/\d^2) node{$\frac{q-1}{d^2}$}

%({\q*\s/(\d*(\d+1))},-0.5) node{$\frac{q-1}{d(d+1)}$}
(-0.75,{\q*\s/(\d*(\d+1))}) node{$\frac{q-1}{d(d+1)}$}

(\s*\q/\d^3,-0.5) node{$\frac{q-1}{d^3}$};

\filldraw
(\s*\q/\d,0) circle(2pt)
(0,\s*\q/\d) circle(2pt)

(0,{\q*\s*(\d+1)/(\d*(\d*\d+1))}) circle(2pt)

({\s*\q*(\d+1)/(\d*\d*\d)},0) circle(2pt)

(\s*\q/\d^2,0) circle(2pt)
(0,\s*\q/\d^2) circle(2pt)

%({\q*\s/(\d*(\d+1))},0) circle(2pt)
(0,{\q*\s/(\d*(\d+1))}) circle(2pt)

(\s*\q/\d^3,0) circle(2pt);

\draw

(\s*\q/\d,0)--(0,\s*\q/\d)
(0,{\q*\s*(\d+1)/(\d*(\d*\d+1))})--({\s*\q*(\d+1)/(\d*\d*\d)},0)
(0,\s*\q/\d^2)--(\s*\q/\d^2,0)
(0,{\q*\s/(\d*(\d+1))})--(\s*\q/\d^2,0)
(0,\s*\q/\d)--(\s*\q/\d^3,0)

(0,{\q*\s/(\d*(\d+1))})--({\s*\q/\d-\q*\s/(\d*(\d+1))},{\q*\s/(\d*(\d+1))})
(0,\s*\q/\d^2)--(\s*\q/\d-\s*\q/\d^2,\s*\q/\d^2);

%({\q*\s/(\d*(\d+1))},0)--(0,\s*\q/\d);

\filldraw[pattern=north east lines, pattern color=yellow]
(0,0)--(0,{\q*\s/(\d*(\d+1))})--(\s*\q/\d^2,0);

\filldraw[pattern=north east lines, pattern color=green]
(0,{\q*\s/(\d*(\d+1))})--(0,\s*\q/\d^2)--({\s*\q/\d^2-\q*\s/(\d*(\d+1))},{\q*\s/(\d*(\d+1))});

\filldraw[pattern=north east lines, pattern color=blue]
(0,\s*\q/\d^2)--(0,{\q*\s*(\d+1)/(\d*(\d*\d+1))})--({(\s*\q/\d-\s*\q/\d^2)/\d^2},\s*\q/\d^2)--({(\s*\q/\d-\q*\s/(\d*(\d+1)))/\d^2},{\q*\s/(\d*(\d+1))});

\draw (1,1) node[circle,draw, thick]{$1$};
\draw (0.6,4.6) node[circle,draw, thick]{$2$};
\draw (1,5.75) node[circle,draw, thick]{$3$};

% a table/legend. 

\begin{scope}[xshift = 0cm,yshift=7cm]

\draw [thick]
(7.5,5.5)--(7.5,2)
(5.5,5)--(12,5);

\draw
(6.5,5.3) node{Pattern}
(10.5,5.3) node{Upper bound on $|Z(f)|$}
(9.1,4.5) node{$\frac{q-1}{d}- \ell= f^\circ$}
(8.7,3.75) node{$d(\ell+g^\circ)$}
(8.9,3) node{$q-1-d^2\ell$}
(11,2.25) node{$d\max\left\{d\left(\ell + g^\circ \right)- \frac{q-1}{d},dg^\circ\right\}$};

\draw
(6.25,4.75)--(6.75,4.75)--(6.75,4.25)--(6.25,4.25)--(6.25,4.75)
(6.25,4)--(6.75,4)--(6.75,3.5)--(6.25,3.5)--(6.25,4)
(6.25,3.25)--(6.75,3.25)--(6.75,2.75)--(6.25,2.75)--(6.25,3.25)
(6.25,2.5)--(6.75,2.5)--(6.75,2)--(6.25,2)--(6.25,2.5);

%\filldraw[pattern=north east lines, pattern color=red]
%(6.25,4.75)--(6.75,4.75)--(6.75,4.25)--(6.25,4.25)--(6.25,4.75);

\filldraw[pattern=north east lines, pattern color=yellow]
(6.25,4)--(6.75,4)--(6.75,3.5)--(6.25,3.5)--(6.25,4);

\filldraw[pattern= north east lines, pattern color=green]
(6.25,3.25)--(6.75,3.25)--(6.75,2.75)--(6.25,2.75)--(6.25,3.25);

\filldraw[pattern=north east lines, pattern color=blue]
(6.25,2.5)--(6.75,2.5)--(6.75,2)--(6.25,2)--(6.25,2.5);

\end{scope}

\end{tikzpicture}
\end{center}
\caption{Bounding $|Z(f)|$ for $f(x) = x^{\frac{q-1}{d}-\ell} +g(x)$.}\label{f1}
\end{figure}

\vspace{0.2cm}
\noindent\textbf{Theorem~\ref{t3}.} \begin{em}Let $\ell \geq 0$ be a nonnegative integer. Suppose $f(x) \in \mathbb{F}_q[x]$ has the form $x^{\frac{q-1}{d} - \ell} + g(x)$, for some $g(x) \in \mathbb{F}_q[x]$ such that $1 \leq g^\circ < \frac{q-1}{d} - \ell$. If one of the following holds, then $|Z(f)| < f^{\circ}$.

\begin{enumerate}[label=\textbf{(\arabic*)}]

\item $d(d+1)\ell + d^2g^\circ < q-1 $;
\item $d^2(\ell + g^\circ) \leq q-1$ and $d(d+1)\ell>q-1$;
\item $d^2(\ell + g^\circ) > q-1$, $d\ell + d^3 g^\circ < q-1$, and $d(d^2+1) \ell + d^3 g^\circ < (q-1)(d+1)$. 
\end{enumerate}\end{em}

\begin{proof}
By Theorem \ref{t1}, we have $|Z(f)| \leq d(\ell+g^\circ)$, which is an improved bound when $d(\ell+g^\circ)<\frac{q-1}{d}-\ell$, i.e. $d(d+1)\ell + d^2g^\circ < q-1$. \\ If $\ell=0$, then obviously (2) and (3) do not hold. In the following discussion, we assume $\ell>0$. Note the proof of Theorem \ref{t1} shows that all nonzero roots of $f(x)$ are roots of
\begin{equation}\label{eqq}
\prod_{\xi \in (\mathbb{F}_q^*)^{\frac{q-1}{d}}} \big(x^\ell g(x) + \xi\big) = x^{d\ell} g^d(x) - 1,    
\end{equation}
and the constant term of $x^{d\ell} g^d(x) - 1$ is $-1$ since $\ell>0$. By using the substitution $x = y^{-1}$  and multiplying by $-y^{d (\ell + g^\circ)}$ in \eqref{eqq}, we see the number of roots of \eqref{eqq} is the same as the number of nonzero roots of the following monic polynomial
\begin{equation}\label{ff1} h(y)=y^{d(\ell + g^\circ)} - y^{dg^\circ}g^d(y^{-1}).\end{equation}
Note that the degree of $h$ is $d(\ell + g^\circ)$, and the degree of $y^{dg^\circ}g^d(y^{-1})$ is $dg^\circ$. Let $\ell'=\frac{q-1}{d}-d(\ell + g^\circ)$. We consider two cases. 
\begin{itemize}
    \item If $\ell'\geq 0$, then we can apply Theorem \ref{t1} to conclude that $$|Z(f)| \leq |Z(h)| \leq d(\ell'+dg^\circ)=q-1-d^2\ell,$$
which is an improved bound provided $\frac{q-1}{d}-d(\ell + g^\circ) \geq 0$ and $q-1-d^2 \ell<\frac{q-1}{d}-\ell$, i.e. $d^2(\ell + g^\circ) \leq q-1$ and $d(d+1)\ell>q-1$.
\item If $\ell'< 0$, then we can apply Theorem \ref{t2} to show that $$|Z(f)| \leq |Z(h)| \leq d \max\left\{-\ell',dg^\circ\right\}=d\max\left\{d\left(\ell + g^\circ \right) - \frac{q-1}{d},dg^\circ\right\},$$
which is an improved bound provided $\frac{q-1}{d}-d(\ell + g^\circ)< 0$ and \\${d\max\left\{d\left(\ell + g^\circ \right) - \frac{q-1}{d},dg^\circ\right\}<\frac{q-1}{d}-\ell}$, i.e.\\ $d^2(\ell+g^\circ)>q-1$, $d^2(\ell+g^\circ)-(q-1)<\frac{q-1}{d}-\ell$ and $d^2 g^\circ<\frac{q-1}{d}-l$.
\end{itemize}
\end{proof}

Below we give examples of polynomials where $|Z(f)| = f^\circ$, showing limitations of extending Theorem~\ref{t3} for a larger range of $g^\circ,\ell$. 

\begin{ex} \label{ex2} \rm Let $D,n \geq 1$ be positive integers such that $D(n+1) | (q-1)$. Then $x^{D(n+1)}-1=0$ has $D(n+1)$ distinct nonzero solutions. Moreover, 
\begin{equation}\label{degex} f(x) = x^{Dn} + x^{D(n-1)} + \cdots + x^D + 1 = \frac{x^{D(n+1)}-1}{x^D-1},\end{equation}
has $Dn$ distinct nonzero solutions, i.e. we have a class of lacunary polynomials $f(x)$ with $|Z(f)| = f^\circ$. We will compare these examples to Theorem~\ref{t3} in the case $n=d=2$. Let
\[ x^{2D} + x^D + 1 = x^{\frac{q-1}{2}-\ell} + g(x),\]
and so $\ell = (q-1)/2 - 2D$ and $g^\circ = D$. Therefore such examples lie on a line in the $\ell,g^\circ$-axis system used above. In Figure~\ref{degtightfig} we illustrate the relation between this line of examples and the regions of improvement.

\begin{figure}[h!]
\begin{center}
\begin{tikzpicture}[scale=0.5]

\def\q{120}
\def\d{2}
\def\s{0.2}

\draw[->,thick]

(-1,0)--(12.5,0);

\draw[->,thick]

(0,-1)--(0,12.5);

\draw
(13,0) node{$g^\circ$}
(0,13) node{$\ell$}

(\s*\q/\d,-0.85) node{$\frac{q-1}{2}$}
(-0.85,\s*\q/\d) node{$\frac{q-1}{2}$}

(-1.25,{\q*\s*(\d+1)/(\d*(\d*\d+1))}) node{$\frac{3(q-1)}{10}$}

({\s*\q*(\d+1)/(\d*\d*\d)},-0.85) node{$\frac{3(q-1)}{8}$}

(\s*\q/\d^2,-0.85) node{$\frac{q-1}{4}$}
(-0.85,\s*\q/\d^2) node{$\frac{q-1}{4}$}

%({\q*\s/(\d*(\d+1))},-0.5) node{$\frac{q-1}{d(d+1)}$}
(-0.85,{\q*\s/(\d*(\d+1))}) node{$\frac{q-1}{6}$}

(\s*\q/\d^3,-0.85) node{$\frac{q-1}{8}$};

\filldraw
(\s*\q/\d,0) circle(2pt)
(0,\s*\q/\d) circle(2pt)

(0,{\q*\s*(\d+1)/(\d*(\d*\d+1))}) circle(2pt)

({\s*\q*(\d+1)/(\d*\d*\d)},0) circle(2pt)

(\s*\q/\d^2,0) circle(2pt)
(0,\s*\q/\d^2) circle(2pt)

%({\q*\s/(\d*(\d+1))},0) circle(2pt)
(0,{\q*\s/(\d*(\d+1))}) circle(2pt)

(\s*\q/\d^3,0) circle(2pt);

\draw

(\s*\q/\d,0)--(0,\s*\q/\d)
(0,{\q*\s*(\d+1)/(\d*(\d*\d+1))})--({\s*\q*(\d+1)/(\d*\d*\d)},0)
(0,\s*\q/\d^2)--(\s*\q/\d^2,0)
(0,{\q*\s/(\d*(\d+1))})--(\s*\q/\d^2,0)
(0,\s*\q/\d)--(\s*\q/\d^3,0)

(0,{\q*\s/(\d*(\d+1))})--({\s*\q/\d-\q*\s/(\d*(\d+1))},{\q*\s/(\d*(\d+1))})
(0,\s*\q/\d^2)--(\s*\q/\d-\s*\q/\d^2,\s*\q/\d^2);

%({\q*\s/(\d*(\d+1))},0)--(0,\s*\q/\d);

\filldraw[pattern=north east lines, pattern color=yellow]
(0,0)--(0,{\q*\s/(\d*(\d+1))})--(\s*\q/\d^2,0);

\filldraw[pattern=north east lines, pattern color=green]
(0,{\q*\s/(\d*(\d+1))})--(0,\s*\q/\d^2)--({\s*\q/\d^2-\q*\s/(\d*(\d+1))},{\q*\s/(\d*(\d+1))});

\filldraw[pattern=north east lines, pattern color=blue]
(0,\s*\q/\d^2)--(0,{\q*\s*(\d+1)/(\d*(\d*\d+1))})--({(\s*\q/\d-\s*\q/\d^2)/\d^2},\s*\q/\d^2)--({(\s*\q/\d-\q*\s/(\d*(\d+1)))/\d^2},{\q*\s/(\d*(\d+1))});

\draw[thick, dashed, red]
(0,\s*\q/\d)--({\s*\q/(\d*\d)},0);

\begin{comment}
% a table/legend. 

\begin{scope}[xshift = 0cm,yshift=7cm]

\draw [thick]
(7.5,5.5)--(7.5,2)
(5.5,5)--(12,5);

\draw
(6.5,5.3) node{Pattern}
(10.5,5.3) node{Upper bound on $|Z(f)|$}
(9.1,4.5) node{$\frac{q-1}{d}- \ell= f^\circ$}
(8.7,3.75) node{$d(\ell+g^\circ)$}
(8.9,3) node{$q-1-d^2\ell$}
(11,2.25) node{$d\max\left\{d\left(\ell + g^\circ \right)- \frac{q-1}{d},dg^\circ\right\}$};

\draw
(6.25,4.75)--(6.75,4.75)--(6.75,4.25)--(6.25,4.25)--(6.25,4.75)
(6.25,4)--(6.75,4)--(6.75,3.5)--(6.25,3.5)--(6.25,4)
(6.25,3.25)--(6.75,3.25)--(6.75,2.75)--(6.25,2.75)--(6.25,3.25)
(6.25,2.5)--(6.75,2.5)--(6.75,2)--(6.25,2)--(6.25,2.5);

%\filldraw[pattern=north east lines, pattern color=red]
%(6.25,4.75)--(6.75,4.75)--(6.75,4.25)--(6.25,4.25)--(6.25,4.75);

\filldraw[pattern=north east lines, pattern color=yellow]
(6.25,4)--(6.75,4)--(6.75,3.5)--(6.25,3.5)--(6.25,4);

\filldraw[pattern= north east lines, pattern color=green]
(6.25,3.25)--(6.75,3.25)--(6.75,2.75)--(6.25,2.75)--(6.25,3.25);

\filldraw[pattern=north east lines, pattern color=blue]
(6.25,2.5)--(6.75,2.5)--(6.75,2)--(6.25,2)--(6.25,2.5);

\end{scope}
\end{comment}

\end{tikzpicture}
\end{center}
\caption{Limitations to improving the degree bound. }\label{degtightfig}
\end{figure}

\end{ex}

It seems plausible that Theorem~\ref{t3} could be improved to include more regions that appear left of the red dashed line in Figure~\ref{degtightfig}. We leave this as an open problem. We conclude this section with several generalizations of the above results.

\begin{thm}\label{ratthm} Let $s(x),t(x),g(x) \in \mathbb{F}_q[x]$ be polynomials such that the rational function defined by $r(x) = s(x)/t(x)$ and $g(x)$ are linearly independent. Also let $h(x) \in \mathbb{F}_q[x]$ be a nonconstant polynomial with no zeros in $\mathbb{F}_q^*$. Define the rational function $f(x) = (h(x))^{\frac{q-1}{d}}r(x) + g(x)$. Then the number of distinct nonzero roots of $f(x)$ is at most $d \max\{s^\circ, g^\circ + t^\circ\}$.
\end{thm}

The ideas behind Theorem~\ref{t1} and Theorem~\ref{t2} can be reused to prove Theorem~\ref{ratthm}. Note that the linearly independent assumption is necessary since we need to ensure $\xi s(x)+t(x)g(x)$ is a nonzero polynomial for $\xi \in (\mathbb{F}_q^*)^{\frac{q-1}{d}}$.  Theorem~\ref{ratthm} is noteworthy since the polynomials $f(x)$ to which it applies may be neither lacunary nor sparse.

The following is a result for $t$-sparse polynomials. It extends Theorem~\ref{t1} and Theorem~\ref{t2}. The proof employs a similar argument to those seen above. Note that Theorem~\ref{intthm} does not require a polynomial with degree close to $\frac{q-1}{d}$.

\begin{thm}\label{intthm} Let 
\[h(x)=\sum_{i=1}^{t} c_ix^{e_i} \in \mathbb{F}_q[x],\]
be a $t$-sparse polynomial. Suppose that there exist integers $a_i,b_i$, $i = 1,2,\ldots,t$ such that
\[e_i=a_i \frac{q-1}{d}+b_i, \text{ where } -\frac{q-1}{d}< b_i <\frac{q-1}{d}.\]
Let $A,B$ be integers such that $\{b_i:1 \leq i \leq t\}$ are contained in the interval $[A,B]$. If $h(x)$ does not vanish on any coset of $(\mathbb{F}_q^\ast)^d$ in $\mathbb{F}_q^\ast$, then $|Z(h)| \leq d(B-A)$. 
\end{thm}

\begin{proof} Let $\{\xi_1,\ldots,\xi_d\} = (\mathbb{F}_q^\ast)^\frac{q-1}{d}$ and define $S_i = \{ a \in \mathbb{F}_q^\ast \colon a^\frac{q-1}{d} = \xi_i\}$. Note that $S_1,\ldots,S_d$ are the cosets of $\mathbb{F}_q^\ast$ of size $\frac{q-1}{d}$. Fix an $i$ in $[1,t]$. For $y \in S_i$ we have 
\begin{equation}\label{inteq} y^{-A}h(y) =y^{-A}\sum_{i=1}^{t} c_iy^{e_i} =\sum_{i=1}^{t} c_i\xi_i^{a_i}y^{b_i-A}. \end{equation}
Observe that above expression is a polynomial in $y$. Moreover, since $h$ does not vanish on any coset of $\mathbb{F}_q^\ast$ of size $\frac{q-1}{d}$, ~\eqref{inteq} is a nonzero polynomial with degree at most $B-A$. Therefore, for each $i$, $h$ has at most $B-A$ zeros in $S_i$. It follows that $|Z(h)| \leq d(B-A)$. 
\end{proof}

We remark that the assumption that $h(x)$ does not vanish on any coset of $(\mathbb{F}_q^\ast)^d$ in $\mathbb{F}_q^\ast$ is often guaranteed. This is the case for the polynomials we will discuss in the remaining of the section. Below we see an example of Theorem~\ref{intthm} in practice.

\begin{ex}\rm
Let $h(x)=\sum_{j=1}^{m} c_jx^{\frac{q-1}{d_j}}+ax+b$, where $d_j \mid (q-1)$ and $d_j<q-1$ . Take $d=\operatorname{lcm} \{d_1,d_2, \ldots, d_m\}$, then it follows that the interval we obtained is $[A,B]=[0,1]$ since $\frac{q-1}{d}$ divides all the exponents expect the linear term and the constant term. So $|Z(h)| \leq d= \operatorname{lcm} \{d_1,d_2, \ldots, d_m\}$.
\end{ex}

Theorem~\ref{t1} and Theorem~\ref{t2} can be obtained from Theorem~\ref{intthm} in the following way. For Theorem \ref{t1}, we can take the interval $[A,B]$ to be $[-\ell, g^\circ]$. For Theorem \ref{t2}, we can take the interval $[A,B]$ to be $[0, \max\{m, g^\circ\}]$. Theorem~\ref{intthm} is strongest when the remainders of the exponents dividing $\frac{q-1}{d}$ are concentrated in a short interval.

The following is a corollary of Theorem~\ref{intthm}. It generalizes Theorem~\ref{t1}. The point is that if a large gap appears between any consecutive exponents, then an improved bound on the number of distinct roots may be possible.

\begin{cor} Let $h(x)=\sum_{i=1}^{t} c_ix^{e_i} \in \mathbb{F}_q[x]$ be a $t$-sparse polynomial, where $\frac{q-1}{d} \geq e_1>e_2> \cdots >e_t$. Define the gap $\delta$ of $h(x)$ to be the largest difference between consecutive exponents, i.e. $\delta=\max\{e_{i-1}-e_i: 2 \leq i \leq t\}$, then $|Z(h)| \leq q-1-d\delta$.
\end{cor}

\begin{proof}
Suppose $\delta=e_{j-1}-e_j$, then the exponents modulo $\frac{q-1}{d}$ are all contained in the interval $[-(\frac{q-1}{d}-e_{j-1}), e_j]$. By Theorem \ref{intthm}, $|Z(h)| \leq d(e_j+\frac{q-1}{d}-e_{j-1})=q-1-d\delta$.
\end{proof}

We remark that Karpinski and Shparlinski's bound $|Z(h)| \leq \frac{t-1}{t} (q-1)$ given in~\cite{KS} can be recovered from the above corollary by taking $d = 1$.

%%%%%%%%%%%%%%%%%%%%%%%%%%%%%

%\clearpage
\section{Iterating to obtain stronger bounds on $|Z(f)|$}\label{seciter}

In this section we build on the ideas presented in the proof of Theorem~\ref{t3}. In particular we iterate the argument of Theorem~\ref{t3} as many times as possible, to yield a stronger bound on $|Z(f)|$. Throughout this section, $d,\ell \geq1$ will be positive integers such that $d|(q-1)$ and $\frac{q-1}{d}-\ell>1$, $g(x) \in \mathbb{F}_q[x]$ will be such that $1 \leq g^\circ < \frac{q-1}{d} - \ell$, $x \nmid g(x)$, and $f(x) \in \mathbb{F}_q[x]$ will be given by
\[ f(x) = x^{\frac{q-1}{d}-\ell} +g(x).\]
Put $\ell_0 = \ell$, and $g_0(x) = g(x)$. For $i \geq 0$, define
\begin{equation}\label{defseq} g_{i+1}(x) = - x^{dg^\circ_i}g^d_i(x^{-1}), \quad \text{and} \quad \ell_{i+1} = \frac{q-1}{d} - d(\ell_i + g_i^\circ).\end{equation}

\begin{lem}\label{manybounds} Let the sequences $\{g_i\}_{i \geq 0}$, $\{\ell_i\}_{i \geq 0}$ be as in~\eqref{defseq}. Suppose that for an integer $k\geq -1$ we have
\[ d(\ell_i+g_i^\circ) < \frac{q-1}{d}, \quad \text{ for } \quad 0 \leq i \leq k.\]
Then 
\[ |Z(f)| \leq \min_{0 \leq i \leq k+1} d(\ell_i+g_i^\circ) .\]

\end{lem}

\begin{proof} 
The case $k=-1$ is Theorem~\ref{t1}. Suppose the hypothesis of the theorem holds for $k \geq 0$. Put
\[ f_i(x) = x^{\frac{q-1}{d}- \ell_i} + g_i(x).\]
Note that $f_0 = f$. Fix any $i$, $0 \leq i \leq k+1$. If $i=0$, then $\ell_i >0$ by definition. For $i \geq 1$, since $d(\ell_{i-1} + g_{i-1}^\circ)< \frac{q-1}{d}$, we again have $\ell_i > 0$. A nonzero root of $f_i(x)$ is a root of $x^{\ell_i}f_i(x)$, and therefore is also a root of 
\[ \prod_{\xi \in (\mathbb{F}_q^*)^{\frac{q-1}{d}}} \big(x^{\ell_i} g_i(x) + \xi\big) = x^{d\ell_i} g^d_i(x) - 1. \]
Substituting $x=y^{-1}$ and multiplying by $-y^{d(\ell_i + g^\circ_i)}$ in the above gives the polynomial
\[ y^{d(\ell_i+g^\circ_i)} - y^{dg_i^\circ} g_i^d(y^{-1}) = f_{i+1}(y). \]
Therefore $|Z(f_{i+1})| \geq |Z(f_i)|$. Since $|Z(f_{i+1})| \leq f^\circ_{i+1}$, we have the desired result. 

\end{proof}

Lemma~\ref{manybounds} potentially provides many upper bounds on $|Z(f)|$. To aid with determining the best bound, we use explicit formulae for the sequences $\{g_i^\circ\}, \{\ell_i+g_i^\circ\}$. 

\begin{lem} \label{formulae} The sequences $\{g_i^\circ\}$ and $\{\ell_i+g^\circ_i\}$ are given by
\[  g_i^\circ = d^i g^\circ , \]
and \
\[  \ell_i + g_i^\circ =\begin{cases} d^i(\ell + g^\circ) -\frac{q-1}{d(d+1)}(d^i -1) & \text{ if } i \text{ is even;}\\  -d^i\ell + \frac{q-1}{d(d+1)}(d^i + 1) & \text{ if } i \text{ is odd.}\end{cases} \]

\end{lem} 

\begin{proof}
From the assumption that $g(x)$ has a nonzero constant term and the recurrence relation $g_{i+1}(x) = - x^{dg^\circ}g^d_i(x^{-1})$, the first statement immediately follows. For the second statement, note that $\ell_{i+1} = \frac{q-1}{d} - d(\ell_i + g_i^\circ)$ implies that
$$
\ell_{i+1}+g_{i+1}^\circ = \frac{q-1}{d} - d\ell_i=\frac{q-1}{d} - d(\ell_i+g_i^\circ)+d^{i+1}g^\circ.
$$
Dividing $d^{i+1}$ on both sides yields
$$
\frac{\ell_{i+1}+g_{i+1}^\circ}{d^{i+1}} = \frac{q-1}{d^{i+2}} - \frac{\ell_i+g_i^\circ}{d^i}+g^\circ.
$$
Now if we set $a_i=\frac{\ell_i+g_i^\circ}{d^i}$ for $i \geq 0$, then we get $a_0=\ell+g^\circ$, and $a_{i+1}+a_i=\frac{q-1}{d^{i+2}}+g^\circ$ for $i \geq 0$. If $i$ is odd, then 
\begin{align*}a_i&=(a_i+a_{i-1})-(a_{i-1}+a_{i-2})+\cdots +(a_1+a_0)-a_0\\
&=-\ell+\frac{(q-1)(1-d)(1+d^{-i})}{d(1-d^2)}=-\ell+\frac{(q-1)(1+d^{-i})}{d(d+1)}.\\
\end{align*}
If $i$ is even, then \begin{align*}a_i&=(a_i+a_{i-1})-(a_{i-1}+a_{i-2})+\cdots +(a_2+a_1)-(a_1+a_0)+a_0\\
&=\ell+g^\circ+\frac{(q-1)(1-d)(1-d^{-i})}{d(1-d^2)}=\ell+g^\circ+\frac{(q-1)(1-d^{-i})}{d(d+1)}.\end{align*} Now applying the relation $\ell_i+g_i^\circ=d^i a_i$ gives the required result. 
\end{proof}

From now on, we will use Lemma~\ref{formulae} without explicitly saying so. If $d = 1$, then the sequence $\{\ell_i + g_i^\circ\}$ oscillates between the values $\ell + g^\circ$ and $q-1-\ell$, and so nothing is gained by considering later terms in $\{d(\ell_i+g_i^\circ)\}$. We will only consider $d\geq 2$. Now better estimates of $|Z(f)|$ may appear later in the sequence $\{d(\ell_i + g_i^\circ)\}$. For example, the first five terms of $\{d(\ell_i + g_i^\circ)\}$ are 
\begin{equation*}
    d(\ell + g^\circ), \ q-1-d^2 \ell,\  d^3(\ell + g^\circ) - (q-1)(d-1),
\end{equation*} 
\begin{equation}\label{fivebounds}
(q-1)(d^2-d+1)-d^4\ell,\ d^5(\ell + g^\circ)-(q-1)(d^3-d^2+d-1).\end{equation}

In Figure~\ref{fiveboundsfig} we illustrate which bound in Equation~\eqref{fivebounds} is best when it is applicable.

\begin{figure}[h!]
\begin{center}
\begin{tikzpicture}[scale=0.8]

\def\q{120}
\def\d{1.6}
\def\s{0.25}

\draw[->,thick]

(-1,0)--(12.5,0);

\draw[->,thick]

(0,-1)--(0,12.5);

\draw
(13,0) node{$g^\circ$}
(0,13) node{$\ell$}

({\s*\q/(\d*\d)},-0.5) node{$\frac{q-1}{d^2}$}

(-0.5,{\s*\q/(\d*\d)}) node{$\frac{q-1}{d^2}$}

(-1.3,{\s*\q*(\d*\d*\d+1)/(\d^4*(\d+1))}) node{$\frac{(q-1)(d^3+1)}{d^4(d+1)}$}

(-0.75,{\q*\s/(\d*(\d+1))}) node{$\frac{q-1}{d(d+1)}$}

(-1.2, {\s*\q*(\d^4-1)/(\d^5*(\d+1))}) node{$\frac{(q-1)(d^4-1)}{d^5(d+1)}$}

(-1.2, {\s*\q*(\d^3-\d)/(\d^4*(\d+1))}) node{$\frac{(q-1)(d^3-d)}{d^4(d+1)}$}

%the xaxis ones

({\q*\s/(\d*(\d+1))},-0.5) node{$\frac{q-1}{d(d+1)}$}

({\s*\q*(\d*\d*\d+1)/(\d^4*(\d+1))},-0.5) node{$\frac{(q-1)(d^3+1)}{d^4(d+1)}$};

\filldraw

({\s*\q/(\d*\d)},0) circle(2pt)
(0,{\s*\q/(\d*\d)}) circle(2pt)
(0,{\s*\q*(\d*\d*\d+1)/(\d^4*(\d+1))}) circle (2pt)
(0,{\q*\s/(\d*(\d+1))}) circle (2pt)
(0, {\s*\q*(\d^4-1)/(\d^5*(\d+1))}) circle (2pt)
(0, {\s*\q*(\d^3-\d)/(\d^4*(\d+1))}) circle (2pt)
%xaxis
({\q*\s/(\d*(\d+1))},0) circle (2pt)
({\s*\q*(\d*\d*\d+1)/(\d^4*(\d+1))},0) circle (2pt);

\draw
({\s*\q/(\d*\d)},0)--(0,{\s*\q/(\d*\d)})
(0,{\q*\s/(\d*(\d+1))})--({\q*\s/(\d*(\d+1))},0)

(0,{\s*\q*(\d*\d*\d+1)/(\d^4*(\d+1))})--({\s*\q*(\d*\d*\d+1)/(\d^4*(\d+1))-\q*\s/(\d*(\d+1))},{\q*\s/(\d*(\d+1))})

(0,{\q*\s/(\d*(\d+1))})--({\s*\q/(\d*\d)},0)

(0,{\q*\s/(\d*(\d+1))})--({\s*\q/(\d*\d)-\q*\s/(\d*(\d+1))},{\q*\s/(\d*(\d+1))})

(0, {\s*\q*(\d^4-1)/(\d^5*(\d+1))})--({\q*\s/(\d*(\d+1))-\s*\q*(\d^4-1)/(\d^5*(\d+1))},{\s*\q*(\d^4-1)/(\d^5*(\d+1))})

(0, {\s*\q*(\d^3-\d)/(\d^4*(\d+1))})--({\q*\s/(\d*(\d+1))-\s*\q*(\d^3-\d)/(\d^4*(\d+1))},{\s*\q*(\d^3-\d)/(\d^4*(\d+1))});

\draw[dashed]
({\s*\q*(\d*\d*\d+1)/(\d^4*(\d+1))-\q*\s/(\d*(\d+1))},{\q*\s/(\d*(\d+1))})--({\s*\q*(\d*\d*\d+1)/(\d^4*(\d+1))},0);

\filldraw[pattern=north east lines, pattern color=yellow]
(0,{\q*\s/(\d*(\d+1))})--({\q*\s/(\d*(\d+1))-\s*\q*(\d^3-\d)/(\d^4*(\d+1))},{\s*\q*(\d^3-\d)/(\d^4*(\d+1))})--(0, {\s*\q*(\d^3-\d)/(\d^4*(\d+1))})--(0,0)--({\s*\q/(\d*\d)},0)--(0,{\q*\s/(\d*(\d+1))});

\filldraw[pattern=north east lines, pattern color=green]
(0,{\s*\q/(\d*\d)})--(0,{\s*\q*(\d*\d*\d+1)/(\d^4*(\d+1))})--({\s*\q*(\d*\d*\d+1)/(\d^4*(\d+1))-\q*\s/(\d*(\d+1))},{\q*\s/(\d*(\d+1))})--({\s*\q/(\d*\d)-\q*\s/(\d*(\d+1))},{\q*\s/(\d*(\d+1))})--(0,{\s*\q/(\d*\d)});

\filldraw[pattern=north east lines, pattern color=purple]
(0,{\q*\s/(\d*(\d+1))})--(0,{\s*\q*(\d*\d*\d+1)/(\d^4*(\d+1))})--({\s*\q*(\d*\d*\d+1)/(\d^4*(\d+1))-\q*\s/(\d*(\d+1))},{\q*\s/(\d*(\d+1))});

\filldraw[pattern=north east lines, pattern color=orange]
(0, {\s*\q*(\d^3-\d)/(\d^4*(\d+1))})--(0, {\s*\q*(\d^4-1)/(\d^5*(\d+1))})--({\q*\s/(\d*(\d+1))-\s*\q*(\d^4-1)/(\d^5*(\d+1))},{\s*\q*(\d^4-1)/(\d^5*(\d+1))})--({\q*\s/(\d*(\d+1))-\s*\q*(\d^3-\d)/(\d^4*(\d+1))},{\s*\q*(\d^3-\d)/(\d^4*(\d+1))});

\filldraw[pattern=north east lines, pattern color=blue]
(0,{\q*\s/(\d*(\d+1))})--(0, {\s*\q*(\d^4-1)/(\d^5*(\d+1))})--({\q*\s/(\d*(\d+1))-\s*\q*(\d^4-1)/(\d^5*(\d+1))},{\s*\q*(\d^4-1)/(\d^5*(\d+1))});

%\draw (1,1) node[circle,draw, thick]{$1$};
%\draw (0.6,4.6) node[circle,draw, thick]{$2$};
%\draw (1,5.75) node[circle,draw, thick]{$3$};

% a table/legend. 

\begin{scope}[xshift = -1cm,yshift=7cm]

\draw [thick]
(7.5,5.5)--(7.5,0)
(5.5,5)--(14,5);

\draw
(6.5,5.3) node{Pattern}
(11.5,5.3) node{Best applicable bound on $|Z(f)|$}
(9.1,4.5) node{$\frac{q-1}{d}- \ell= f^\circ$}
(8.7,3.75) node{$d(\ell+g^\circ)$}
(9,3) node{$q-1-d^2\ell$}
(10.75,2.25) node{$d^3(\ell + g^\circ) - (q-1)(d-1)$}
(10.5,1.5) node{$(q-1)(d^2-d+1)-d^4\ell$}
(11.8,0.75) node{$d^5(\ell + g^\circ)-(q-1)(d^3-d^2+d-1)$};

\draw
(6.25,4.75)--(6.75,4.75)--(6.75,4.25)--(6.25,4.25)--(6.25,4.75)
(6.25,4)--(6.75,4)--(6.75,3.5)--(6.25,3.5)--(6.25,4)
(6.25,3.25)--(6.75,3.25)--(6.75,2.75)--(6.25,2.75)--(6.25,3.25)
(6.25,2.5)--(6.75,2.5)--(6.75,2)--(6.25,2)--(6.25,2.5)
(6.25,1.75)--(6.75,1.75)--(6.75,1.25)--(6.25,1.25)--(6.25,1.75)
(6.25,1)--(6.75,1)--(6.75,0.5)--(6.25,0.5)--(6.25,1);

%\filldraw[pattern=north east lines, pattern color=red]
%(6.25,4.75)--(6.75,4.75)--(6.75,4.25)--(6.25,4.25)--(6.25,4.75);

\filldraw[pattern=north east lines, pattern color=yellow]
(6.25,4)--(6.75,4)--(6.75,3.5)--(6.25,3.5)--(6.25,4);

\filldraw[pattern= north east lines, pattern color=green]
(6.25,3.25)--(6.75,3.25)--(6.75,2.75)--(6.25,2.75)--(6.25,3.25);

\filldraw[pattern=north east lines, pattern color=orange]
(6.25,2.5)--(6.75,2.5)--(6.75,2)--(6.25,2)--(6.25,2.5);

\filldraw[pattern=north east lines, pattern color=purple]
(6.25,1.75)--(6.75,1.75)--(6.75,1.25)--(6.25,1.25)--(6.25,1.75);

\filldraw[pattern=north east lines, pattern color=blue]
(6.25,1)--(6.75,1)--(6.75,0.5)--(6.25,0.5)--(6.25,1);

\end{scope}

\end{tikzpicture}
\end{center}
\caption{Comparing the six bounds in \eqref{fivebounds} for $|Z(f)|$}\label{fiveboundsfig}
\end{figure}
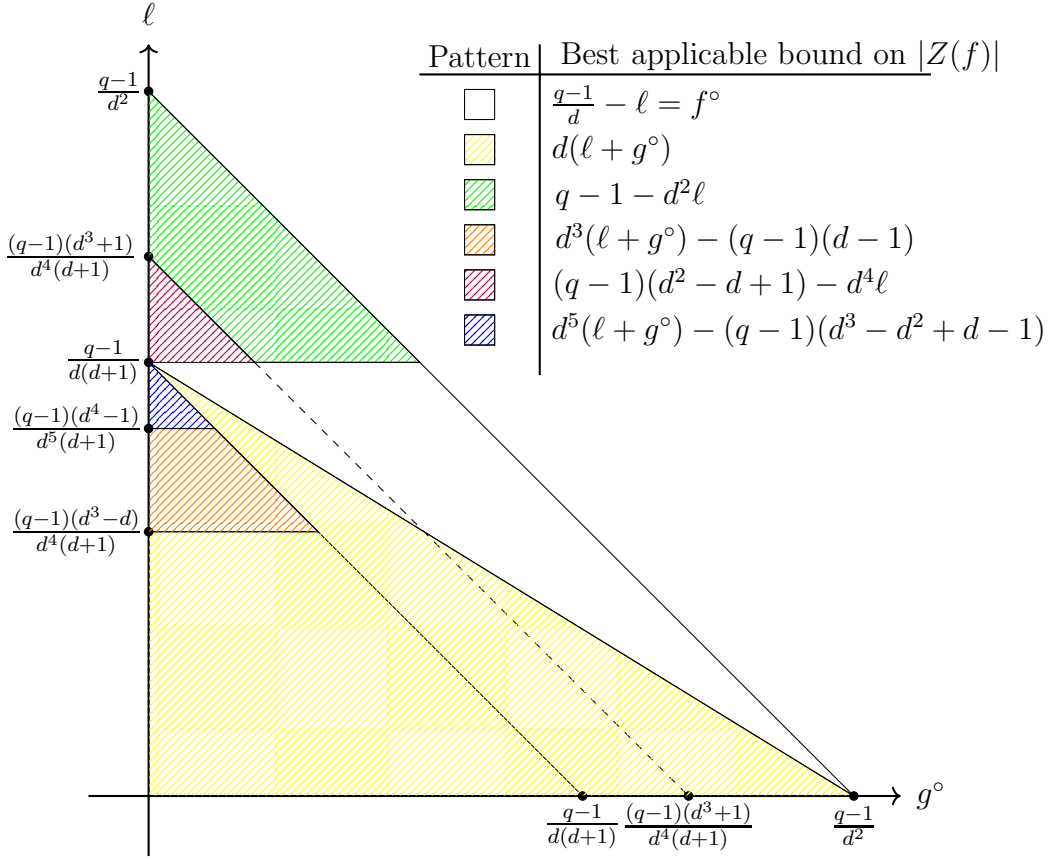

The formulae in Lemma \ref{formulae} can be used to determine the minimum $d(\ell_i + g_i^\circ)$.

\begin{lem} \label{compare} \begin{enumerate}[label=\textbf{(\arabic*)}]
\item If $\ell > \frac{q-1}{d(d+1)}$, then for all $i > j \geq 0$ and $t \geq 0$,
\[ \ell_{2i+1} + g^\circ_{2i+1} < \ell_{2j+1} + g_{2j+1}^\circ < \frac{q-1}{d}-\ell , \quad  \text{and} \quad \ell_{2i+1} + g^\circ_{2i+1} < \ell_{2t} + g_{2t}^\circ.  \]

\item If $\ell +g^\circ < \frac{q-1}{d(d+1)}$, then for all $i > j \geq 0$ and $t \geq 0$,
\[ \ell_{2i} + g^\circ_{2i} < \ell_{2j} + g_{2j}^\circ < \frac{q-1}{d}-\ell , \quad  \text{and} \quad \ell_{2i} + g^\circ_{2i} < \ell_{2t+1} + g_{2t+1}^\circ.  \]
 
\item If $\ell \leq  \frac{q-1}{d(d+1)}$, and $\ell +g^\circ \geq \frac{q-1}{d(d+1)}$, then for all $i > j \geq 0$,
\[ \ell_{2i} + g^\circ_{2i} \geq \ell_{2j} + g_{2j}^\circ , \quad  \text{and} \quad \ell_{2i+1} + g^\circ_{2i+1} \geq \ell_{2j+1} + g_{2j+1}^\circ.  \]

\end{enumerate}

\end{lem}

\begin{proof}
By Lemma \ref{formulae}, for each $i \geq 0$, we have
$$\ell_{2i} +g_{2i}^\circ=\frac{q-1}{d(d+1)}+d^{2i}\bigg(\ell+g^\circ-\frac{q-1}{d(d+1)}\bigg), \quad \ell_{2i+1} + g_{2i+1}^\circ=\frac{q-1}{d(d+1)}+d^{2i+1}\bigg(\frac{q-1}{d(d+1)}-\ell\bigg).$$
To determine the monotonicity of the two sequences, it suffices to compare the size of $\ell+g^\circ,\ell$ and $\frac{q-1}{d(d+1)}$. Therefore, there are the following three cases:
\begin{enumerate}[label=\textbf{(\arabic*)}]
    \item  If $\ell > \frac{q-1}{d(d+1)}$, then the sequence $\{\ell_{2i} + g_{2i}^\circ\}$ is strictly increasing, the sequence $\{\ell_{2i+1} +g_{2i+1}^\circ\}$ is strictly decreasing, and $\ell_{1} + g_{1}^\circ<\ell_0+g_0^\circ$.
    \item If $\ell +g^\circ < \frac{q-1}{d(d+1)}$, then the sequence $\{\ell_{2i} + g_{2i}^\circ\}$ is strictly decreasing, the sequence $\{\ell_{2i+1} +g_{2i+1}^\circ\}$ is strictly increasing, and $\ell_{1} + g_{1}^\circ>\ell_0+g_0^\circ$.
    \item If $\ell \leq  \frac{q-1}{d(d+1)}$, and $\ell +g^\circ \geq \frac{q-1}{d(d+1)}$, then both sequences $\{\ell_{2i} + g_{2i}^\circ\}$ and $\{\ell_{2i+1} +g_{2i+1}^\circ\}$ are increasing.
\end{enumerate}
\end{proof}

It remains to understand the inequalities $d(\ell_i + g_i^\circ) \leq \frac{q-1}{d}$. Using the formulae in Lemma \ref{formulae}, we see that for $i \geq 0$, $d(\ell_{2i} + g_{2i}^\circ) \leq \frac{q-1}{d}$ is equivalent to 
\begin{equation*}\label{evenc} \ell + g^\circ \leq (q-1) \left( \frac{1+ d^{-2i-1}}{d(d+1)}\right) .\end{equation*}
And for $i \geq 0$, $d(\ell_{2i+1} + g_{2i+1}^\circ) \leq \frac{q-1}{d}$ is equivalent to 
\begin{equation*}\label{oddc} \ell \geq (q-1) \left( \frac{1- d^{-2i-2}}{d(d+1)}\right) .\end{equation*}

The following theorem puts the above discussion together. We use Lemma~\ref{compare} to determine what bound from Lemma~\ref{manybounds} is best. Now we are ready to prove our main result, Theorem \ref{bestboundthm}. We restate the theorem for convenience. 

\vspace{0.2cm}

\noindent\textbf{Theorem~\ref{bestboundthm}.}\begin{em} Let $f(x) \in \mathbb{F}_q[x]$ be as in~\eqref{fform} and assume that the constant term of $f$ is nonzero. Then exactly one of the following holds.  
\begin{enumerate}[label=\textbf{(\arabic*)}]
\item If $\ell > \frac{q-1}{d(d+1)}$ and $i \geq -1$ is the largest integer such that 
\begin{equation}\label{1eq} \ell + g^\circ < (q-1) \left( \frac{1+ d^{-2i-1}}{d(d+1)}\right),\end{equation}
then 

\[ |Z(f)| \leq \frac{q-1}{d+1} - d^{2i+2}\left(\ell - \frac{q-1}{d(d+1)}\right).\]

\item If $\ell + g^\circ  < \frac{q-1}{d(d+1)}$ and $i \geq -1$ is the largest integer such that 
\begin{equation}\label{2eq} \ell > (q-1) \left( \frac{1- d^{-2i-2}}{d(d+1)}\right),\end{equation}
then 

\[ |Z(f)| \leq \frac{q-1}{d+1} - d^{2i+3} \left( \frac{q-1}{d(d+1)} - (\ell + g^\circ) \right).\]

\item  If $\ell \leq \frac{q-1}{d(d+1)}$, $ \ell + g^\circ \geq \frac{q-1}{d(d+1)}$, and $d(d+1)\ell + d^2 g^\circ < q-1$, then 
\[ |Z(f)| \leq d(\ell+g^\circ) .\]

\item If $\ell \leq \frac{q-1}{d(d+1)}$, $ \ell + g^\circ \geq \frac{q-1}{d(d+1)}$, and $d(d+1)\ell + d^2 g^\circ \geq q-1$, then 
\[ |Z|(f)| \leq f^\circ = \frac{q-1}{d} - \ell .\]

\end{enumerate}\end{em}

\begin{proof}
\vspace{0.2cm}

\textit{\textbf{(1)}} The condition $\ell > \frac{q-1}{d(d+1)}$ gives that $d(\ell_j+g_j^\circ) < \frac{q-1}{d}$ for all odd $j \geq 1$. Equation~\eqref{1eq} implies that $d(\ell_j+g_j^\circ) < \frac{q-1}{d}$ for all even $0 \leq j \leq  2i$. By Lemma~\ref{manybounds}, $|Z(f)| \leq d(\ell_j + g^\circ_j)$ for $0 \leq j \leq 2i+1$. By Lemma~\ref{compare}, the lowest upper bound of this set is 
\[ d(\ell_{2i+1} + g_{2i+1}^\circ) = \frac{q-1}{d+1}\left(d^{2i+1}+1 \right) - d^{2i+2}\ell.\]
%\vspace{0.2cm}

\textit{\textbf{(2)}} The condition $\ell + g^\circ < \frac{q-1}{d(d+1)}$ gives that $d(\ell_j+g_j^\circ) < \frac{q-1}{d}$ for all even $j \geq 0$. Equation~\eqref{2eq} implies that $d(\ell_j+g_j^\circ) < \frac{q-1}{d}$ for all odd $1 \leq j \leq  2i+1$. By Lemma~\ref{manybounds}, $|Z(f)| \leq d(\ell_j + g^\circ_j)$ for $0 \leq j \leq 2i+2$. By Lemma~\ref{compare}, the lowest upper bound of this set is 
\[ d(\ell_{2i+2} + g_{2i+2}^\circ) = d^{2i+3}(\ell + g^\circ) - \frac{q-1}{d+1} \left( d^{2i+2}-1 \right).\]

%\vspace{0.2cm}

\textit{\textbf{(3\&4)}} If $\ell \leq \frac{q-1}{d(d+1)}$ and $ \ell + g^\circ \geq \frac{q-1}{d(d+1)}$, then by Lemma~\ref{compare}, $\{d(\ell_{2i} + g^\circ_{2i})\}_{i \geq 0}$ is increasing. Similarly, $\{d(\ell_{2i+1} + g^\circ_{2i+1})\}_{i \geq 0}$ is increasing and $f^\circ < d(\ell_1 + g^\circ_1)$. Therefore either $f^\circ$ or $d(\ell + g^\circ)$ is the lowest upper bound on $|Z(f)|$, and the remaining two cases follow immediately.

\end{proof}

To illustrate how the bound improves over the iteration employed in Theorem~\ref{bestboundthm}, consider the difference between the bound on $|Z(f)|$ given in part (2) of Theorem~\ref{bestboundthm} and the degree bound. 
\[ \frac{q-1}{d}-\ell - \bigg( \frac{q-1}{d+1} - d^{2i+3} \left( \frac{q-1}{d(d+1)} - (\ell + g^\circ) \right) \bigg)\] 
\[ =\left( 1 + d^{2i+3} \right) \left( \frac{q-1}{d(d+1)} - (\ell + g^\circ) \right) + g^\circ. \]
In other words, the difference in the degree bound and the iterative bound grows exponentially in the number of iterations. 

In Example \ref{ex} and in Example \ref{ex2}, we saw that the bound in cases (3) and (4) of Theorem~\ref{bestboundthm} can be tight. The following is an example where the bound in case (1) of Theorem~\ref{bestboundthm} is tight. 

\begin{ex} \rm
Let $p=379$, $d=2$, $\ell=\frac{p-7}{4}=93,g^\circ=1$, and $f(x)=x^{96}+x+317 \in \mathbb{F}_p[x]$. Below we give $f_i(x) = x^{\frac{q-1}{d}-\ell_i} + g_i(x)$ for $i=1,2$. These are the polynomials formed in the iteration technique and have the property $|Z(f)| \leq |Z(f_i)|$. 
\[f_1(x)=x^{188}-54x^2-255x-1, \quad f_2(x) =x^6+378x^4+248x^3+55x^2+127x+116.\]

Therefore $|Z(f)| \leq |Z(f_1)| \leq |Z(f_2)| \leq 6$. Note that we can also bound $|Z(f)|$ by applying Theorem~\ref{bestboundthm} which gives $|Z(f)| \leq p-1-d^2 \ell=6$. Indeed, we can verify that $Z(f)=\{21,37,89,303,322,365\}$, so the iterative technique gives a tight bound in this case. 
\end{ex}

Below is an example where the bound in case (2) of Theorem~\ref{bestboundthm} is close to tight.

\begin{ex}\rm
Let $p=367$, $d=2$, $\ell=\frac{p+1}{8}=46,g^\circ=1$, and $f(x)=x^{137}+x+111 \in \mathbb{F}_p[x]$. By Theorem~\ref{bestboundthm}, we can take $i=1$ and get 
$|Z(f)| \leq d^3(\ell + g^\circ) - (p-1)(d-1)=8(\ell+g^\circ)-(p-1)=10.$
We can verify that $Z(f)=\{82,105,109,195,216,246,333\}$.
\end{ex}

Note that in the two above examples, the degree bound on $|Z(f)|$ and the bound obtained by earlier stages of the iteration are very far from the true size of $|Z(f)|$. This demonstrates the effectiveness of the iterative technique.

%%%%%%%%%%%%%%%%%%%%%%%%%%%%%

\section{Concluding remarks}\label{seccon}

The main theorem of Section 4, Theorem~\ref{bestboundthm} was proved by iterating Theorem~\ref{t1} as many times as possible. A more complicated procedure involving both Theorem~\ref{t1} and Theorem~\ref{t2} can be used in some cases to obtain bounds on $|Z(f)|$ for more pairs of $g^\circ,\ell$. Indeed, this idea was partly used in the proof of Theorem~\ref{t3}. 

We expect there are ways to extend some of our results. We propose the following problem, which would extend Theorem~\ref{t3}. 

\begin{prob} For what pairs of $g^\circ,\ell$ can the degree bound on $|Z(f)|$ in~\eqref{fform} be improved? 
\end{prob}

%%%%%%%%%%%%%%%%%%%%%%%%%%%%%

\section{Acknowledgements} The authors would like to thank Greg Martin for helpful discussions. The authors are also indebted to Igor Shparlinski for suggesting Corollary~\ref{sqrtcor} and Theorem~\ref{ratthm} in addition to other insightful comments. The research of the first author was supported in part by NSERC Discovery, OTKA K 119528 and NKFI KKP 133819 grants. The research of the second author was supported in part by Killam and NSERC doctoral scholarships.

%%%%%%%%%%%%%%%%%%%%%%%%%%%%%


\begin{thebibliography}{1}


\bibitem{BB} S. Ball, A. Blokhuis, Lacunary polynomials over finite fields, in G. Mullen, D. Panario (Eds.), \emph{Handbook of Finite Fields}, (2013), pp. 556-562. CRC Press. 

\bibitem{BS}
K. Bibak, I. Shparlinski,
On Fully Split Lacunary Polynomials in Finite Fields,
\emph{Bulletin Polish Acad. Sci. Math.} 59 (2011), 197-202.

\bibitem{CF}
R. Canetti, J. Friedlander, S. Konyagin, M. Larsen, D. Lieman, I. Shparlinski,
On the statistical properties of Diffie-Hellman distributions,
\emph{Israel J. Math. }120 (2000), 23-46.


\bibitem{KS} 
M. Karpinski, I. Shparlinski,
On some approximation problems concerning sparse polynomials over finite fields,
\emph{Theoretical Computer Science}, 157 (1996), 259-266.



\bibitem{K}
Z. Kelley,
Roots of Sparse Polynomials over a Finite Field,
\emph{LMS J. Comput. Math.} 19 (2016), 196-204.


\bibitem{KO}
Z. Kelley, S. Owen,
Estimating the Number Of Roots of Trinomials over Finite Fields,
\emph{Journal of Symbolic Computation}, 79 (2017), 108-118.



\bibitem{r} L. R\'edei, ``L\"uckenhafte Polynome \"uber endlichen K\"orperrn,'' Birkh\"auser, Basel, 1970 (Engl. trans. Lacunary Polynomials over Finite Fields, North Holland, Amsterdam, 1973). 

\end{thebibliography}
\end{document}